\documentclass[english]{article}
\usepackage[charter]{mathdesign}
\usepackage[T1]{fontenc}
\usepackage[latin9]{inputenc}
\usepackage[letterpaper]{geometry}
\usepackage{fancyhdr}
\usepackage{babel}
\usepackage{mathtools}
\usepackage{amsmath}
\usepackage{amsthm}
\usepackage[numbers]{natbib}
\usepackage[unicode=true,
 bookmarks=true,bookmarksnumbered=false,bookmarksopen=false,
 breaklinks=false,pdfborder={0 0 1},backref=false,colorlinks=false]
 {hyperref}

\makeatletter
\theoremstyle{plain}
\newtheorem{thm}{\protect\theoremname}
  \theoremstyle{definition}
  
  \theoremstyle{remark}
  \newtheorem{rem}[thm]{\protect\remarkname}
  \theoremstyle{plain}
  \newtheorem{cor}[thm]{\protect\corollaryname}
  \theoremstyle{definition}
  \newtheorem{example}[thm]{\protect\examplename}

\date{}

\usepackage[none]{hyphenat}

\makeatother

  \providecommand{\corollaryname}{Corollary}
  \providecommand{\definitionname}{Definition}
  \providecommand{\examplename}{Example}
  \providecommand{\remarkname}{Remark}
\providecommand{\theoremname}{Theorem}

\begin{document}
\lhead{Elliptic Curves with Large Intersection of Projective Torsion Points}\rhead{Bogomolov $\cdot$ Fu}

\title{\textsc{Elliptic Curves with Large Intersection}\\
\textsc{ of Projective Torsion Points}}

\author{\textsc{Fedor Bogomolov $\cdot$ Hang Fu}}
\maketitle
\begin{quote}
\textbf{\small{}Abstract.}{\small{} We construct pairs
of elliptic curves over number fields with large intersection
of projective torsion points.}{\small \par}

\textbf{\small{}Keywords.}{\small{} Elliptic curves $\cdot$ Torsion points $\cdot$ Division
polynomials $\cdot$ Unlikely intersections}{\small \par}

\textbf{\small{}Mathematics Subject Classification.}{\small{} 14H52 $\cdot$ 14Q05}{\small \par}
\end{quote}

\section{Introduction}

Let $E$ be an elliptic curve over a number field, 
\begin{itemize}
\item $E[n]\subseteq E(\bar{\mathbb Q})$ its $n$-torsion subgroup,
\item $E^{*}[n]\subseteq E[n]$ the collection of torsion points of 
order \textit{exactly} $n$, 
\item 
$E[\infty]=\cup_{n=1}^{\infty}E[n]$
the collection of all torsion points, and 
\item $\pi:E\to\mathbb{P}^{1}$ the standard double cover, 
identifying $\pm P\in E$.
\end{itemize}

\medskip{}

In \cite{MR3536148}, the authors considered intersections
$$
\pi_{1}(E_{1}[\infty])\cap \pi_{2}(E_{2}[\infty]),
$$
for distinct elliptic curves
$E_{1}$ and $E_{2}$. They observed the equivalence of conditions
\begin{itemize}
\item[(A)]
$\pi_{1}(E_{1}[2])=\pi_{2}(E_{2}[2])$,
\item[(B)] 
$\pi_{1}(E_{1}[\infty])=\pi_{2}(E_{2}[\infty])$, and 
\item[(C)] 
$\#\pi_{1}(E_{1}[\infty])\cap\pi_{2}(E_{2}[\infty])=\infty$
\end{itemize}
and found pairs of elliptic curves $(E_{1}, E_{2})$ 
failing these conditions but satisfying 
$$
\#\pi_{1}(E_{1}[\infty])\cap\pi_{2}(E_{2}[\infty])\geq 14.
$$
We view results of this type as instances of {\em unlikely intersections}, 
see, e.g., \cite{MR2918151} for an extensive study of related problems.
In this note, we provide some improvements of  
\cite{MR3536148} and \cite{MR2349648}:

\begin{thm}
\label{Theorem 1}
\[
\underset{\{(E_{1},E_{2}):\pi_{1}(E_{1}[2])\neq\pi_{2}(E_{2}[2])\}}{\textup{sup}}\#\pi_{1}(E_{1}[\infty])\cap\pi_{2}(E_{2}[\infty])\geq 22.
\]
\end{thm}

The proof will be given in Section \ref{Section 2}. 

\medskip{}

In \cite{MR2349648}, 
the first author and Tschinkel studied 
{\em intrinsic} properties of subsets of $\mathbb P^1(\bar{\mathbb Q})$ 
generated by images of torsion points. 
Starting with 
distinct points $a,b,c,d\in\mathbb{P}^{1}(\bar{\mathbb{Q}})$, they
defined 
$L[a,b,c,d]\supseteq\{a,b,c,d\}$ as the smallest subset of 
$\mathbb{P}^{1}(\bar{\mathbb{Q}})$ such
that if 
$\pi(E[2])\subseteq L[a,b,c,d]$, then $\pi(E[\infty])\subseteq L[a,b,c,d]$.
They proved:
\begin{thm}
\label{Theorem 3}If $\delta\not\in\{0,\infty,\pm1,\pm i\}$, then
$L[\pm\delta^{\pm1}]\backslash\{\infty\}$ is a field. Moreover, it is 
closed under taking square roots.
\end{thm}
A shortened version of their proof 
can be found in \cite{2017}. 
In Section \ref{Section 3}, we establish
new properties of this construction.

\section{\label{Section 2}Large Intersection}

In \cite{MR3536148}, we considered torsion points
of order $3$ and $5$. Here, we will instead focus on
torsion points of order $3$ and $p$, where $p$ is a different
prime number. 
The calculations below were performed with Mathematica 11.0
\cite{2016}.

\begin{proof}
[Proof of Theorem \ref{Theorem 1}]Consider the family
\[
E_{\delta}:y^{2}=x^{4}-\left(\delta^{2}+\frac{1}{\delta^{2}}\right)x^{2}+1,
\]
which is a curve of genus $1$ with a unique singularity at $(0:1:0)$,
provided $\delta^{4}\neq0,1$. Fix $O_{\delta}=(\delta,0)$ and 
$\hat{\pi}:\mathbb{P}^{2}\to\mathbb{P}^{1},(x,y)\mapsto x$.
Recall that 
$$
\hat{\pi}(E_{\delta}[2])=\{\pm\delta^{\pm1}\},\quad 
\hat{\pi}(E_{\delta}^{*}[4])=\{0,\infty,\pm1,\pm i\},
$$
and that if $a\in\hat{\pi}(E_{\delta}[\infty])$, then $-a,1/a,-1/a\in\hat{\pi}(E_{\delta}[\infty])$.
The nonsingular model of $E_{\delta}$ is
\[
E_{\delta}^{ns}:Y^{2}=X(X-1)\left(X-\frac{1}{4}\left(\delta+\frac{1}{\delta}\right)^{2}\right),
\]
where
\[
X=\frac{(\delta^{2}+1)(\delta x-1)}{2\delta(x-\delta)},Y=\frac{(\delta^{4}-1)y}{4\delta(x-\delta)^{2}}.
\]
From the division polynomials {[}\citealp{MR0518817}, Chapter II; \citealp{MR2514094}, Page 105, Exercise 3.7{]}
of $E_{\delta}^{ns}$, we know that the third and seventh (modified)
division polynomials of $E_{\delta}$ are
\begin{alignat*}{2}
F_{3}(x,\delta) & \quad & \coloneqq\quad & \prod_{\{\hat{\pi}(P):P\in E_{\delta}^{*}[3]\}}(x-\hat{\pi}(P))\cdot\delta\\
 & \quad & =\quad & 2x^{3}\delta^{2}+(x^{4}-1)\delta-2x,\\
F_{7}(x,\delta) & \quad & \coloneqq\quad & {\displaystyle \prod_{\{\hat{\pi}(P):P\in E_{\delta}^{*}[7]\}}}(x-\hat{\pi}(P))\cdot\delta^{6}\\
 & \quad & =\quad & 64x^{14}\delta^{12}+\cdots+(x^{24}-14x^{20}-305x^{16}-644x^{12}-305x^{8}-14x^{4}+1)\delta^{6}+\cdots+64x^{10}.
\end{alignat*}
Now we want to find $\delta_{1}$ and $\delta_{2}$ such that there
exist 
$$
u\in\hat{\pi}(E_{\delta_{1}}^{*}[3])\cap\hat{\pi}(E_{\delta_{2}}^{*}[3])\quad \text{ and } \quad
v\in\hat{\pi}(E_{\delta_{1}}^{*}[7])\cap\hat{\pi}(E_{\delta_{2}}^{*}[7]).
$$
In other words, 
$$
F_{3}(u,\delta_{1})=F_{3}(u,\delta_{2})=0\quad \text{  and } \quad 
F_{7}(v,\delta_{1})=F_{7}(v,\delta_{2})=0.
$$
Since $F_{3}(x,\delta)$ is a quadratic polynomial in $\delta$, any
fixed $u$ such that $u^{4}\neq0,1$ and $u^{8}+14u^{4}+1\neq0$ gives
exactly two roots satisfying $\delta_{1}^{4},\delta_{2}^{4}\neq0,1$
and $\delta_{1}\neq\pm\delta_{2}^{\pm1}$. Since $\delta_{1}$ and
$\delta_{2}$ are also two roots of $F_{7}(v,\delta)=0$, we need
$F_{3}(u,\delta)$ divides $F_{7}(v,\delta)$ as polynomials in $\delta$.
By long division, this is equivalent to require both $C_{7,0}(u,v)=0$
and $C_{7,1}(u,v)=0$, where $C_{7,0}$ is a polynomial of degree
$58$ with $177$ terms, and $C_{7,1}$ is a polynomial of degree
$62$ with $202$ terms. The resultant of $C_{7,0}$ and $C_{7,1}$ w.r.t. $v$ is
\[
2^{240}u^{900}\left(u^{4}-1\right)^{132}\left(128u^{48}+24352u^{44}+\cdots+24352u^{4}+128\right)\left(u^{72}+16u^{68}+\cdots+16u^{4}+1\right)^{3},
\]
whose roots are the $u$-coordinates of their common points. The resultant of $C_{7,0}$ and $C_{7,1}$ w.r.t. $u$ is
\[
2^{320}v^{692}\left(v^{4}-1\right)^{132}\left(8v^{48}+4776v^{44}+\cdots+4776v^{4}+8\right)\left(v^{216}+690v^{212}+\cdots+690v^{4}+1\right),
\]
whose roots are the $v$-coordinates of their common points. Here
the cube power in the first resultant plays the crucial role. There
are $48+72=120$ nontrivial $u$-coordinates, and $48+216=264$ nontrivial
$v$-coordinates, hence there exists some $u$ that is shared by at
least three common points $(u,v_{k}),k=1,2,3$. Thus we have
\[
\begin{cases}
\hat{\pi}(E_{\delta_{1}}[2])\neq\hat{\pi}(E_{\delta_{2}}[2]),\hat{\pi}(E_{\delta_{1}}^{*}[4])=\hat{\pi}(E_{\delta_{2}}^{*}[4])=\{0,\infty,\pm1,\pm i\},\\
u\in\hat{\pi}(E_{\delta_{1}}^{*}[3])\cap\hat{\pi}(E_{\delta_{2}}^{*}[3]),-u,1/u,-1/u\in\hat{\pi}(E_{\delta_{1}}^{*}[6])\cap\hat{\pi}(E_{\delta_{2}}^{*}[6]),\\
v_{k}\in\hat{\pi}(E_{\delta_{1}}^{*}[7])\cap\hat{\pi}(E_{\delta_{2}}^{*}[7]),-v_{k},1/v_{k},-1/v_{k}\in\hat{\pi}(E_{\delta_{1}}^{*}[14])\cap\hat{\pi}(E_{\delta_{2}}^{*}[14]),k=1,2,3,
\end{cases}
\]
so the supremum is at least $22$.
\end{proof}
\begin{rem}
\label{Remark 4}The same phenomenon continues for larger primes $p=11$,
$13$, and $17$. To accelerate our computer-aided calculation, we
note that $C_{p,0}(u,v)$ and $C_{p,1}(u,v)$ are essentially polynomials
in $u^{4}$ and $v/u$. When $p=11$,
\[
\begin{cases}
\text{the resultant of }C_{11,0}\text{ and }C_{11,1}\text{ w.r.t. }v/u=(\text{trivial factors})(\text{a factor of degree }12)^{3}\\
\cdot(\text{a factor of degree }30)(\text{a factor of degree }150)^{3}\in\mathbb{Z}[u^{4}],\\
\\
\text{the resultant of }C_{11,0}\text{ and }C_{11,1}\text{ w.r.t. }u^{4}=(\text{trivial factors})\\
\cdot(\text{a factor of degree }12)(\text{a factor of degree }24)\\
\cdot(\text{a factor of degree }30)(\text{a factor of degree }450)\in\mathbb{Z}[v/u];
\end{cases}
\]
when $p=13$,
\[
\begin{cases}
\text{the resultant of }C_{13,0}\text{ and }C_{13,1}\text{ w.r.t. }v/u=(\text{trivial factors})(\text{a factor of degree }18)^{3}\\
\cdot(\text{a factor of degree }42)(\text{a factor of degree }324)^{3}\in\mathbb{Z}[u^{4}],\\
\\
\text{the resultant of }C_{13,0}\text{ and }C_{13,1}\text{ w.r.t. }u^{4}=(\text{trivial factors})\\
\cdot(\text{a factor of degree }18)(\text{a factor of degree }36)\\
\cdot(\text{a factor of degree }42)(\text{a factor of degree }972)\in\mathbb{Z}[v/u];
\end{cases}
\]
when $p=17$,
\[
\begin{cases}
\text{the resultant of }C_{17,0}\text{ and }C_{17,1}\text{ w.r.t. }v/u=(\text{trivial factors})(\text{a factor of degree }54)^{3}\\
\cdot(\text{a factor of degree }72)(\text{a factor of degree }1008)^{3}\in\mathbb{Z}[u^{4}],\\
\\
\text{the resultant of }C_{17,0}\text{ and }C_{17,1}\text{ w.r.t. }u^{4}=(\text{trivial factors})\\
\cdot(\text{a factor of degree }54)(\text{a factor of degree }54)(\text{a factor of degree }54)\\
\cdot(\text{a factor of degree }72)(\text{a factor of degree }3024)\in\mathbb{Z}[v/u].
\end{cases}
\]
Here ``a factor'' always means an irreducible factor. From these
facts, we note that for $p=11$, $13$, and $17$, the resultant of
$C_{p,0}$ and $C_{p,1}$ with respect to $v/u$ only involves factors
of multiplicity $1$ and $3$, the former has degree $(p^{2}-1)/4$,
while the latter suggests that as $p\to\infty$, there might be infinitely
many pairs $E_{\delta_{1}}$ and $E_{\delta_{2}}$ with at least $22$
intersection points, and this record might not be broken by the current
approach. Due to the computational limitation, we have no information
for $p\geq19$.
\end{rem}

\begin{rem}
\label{Remark 5}
Let 
$$
J_{k}(n):=n^{k}\prod_{p|n}(1-p^{-k})
$$
denote Jordan's totient function. 
In Appendix of \cite{MR3536148}, the authors attempted
to show that the values $J_{2}(n)$, $J_{4}(n)$, and $J_{6}(n)$ would suffice to
determine $n$. However, as indicated by Prof. Kevin
Ford in a personal email communication, this would contradict 
a plausible conjecture in analytic number theory.
\end{rem}

\section{\label{Section 3}The Fields Generated by the Projective Torsion
Points}

For simplicity, we write $L_{\delta}=L[\pm\delta^{\pm1}]$ and $L_{\mathbb{Q}}=L[0,\infty,1,-1]$.
We have seen that 
$$
\hat{\pi}(E_{\delta}[2])=\{\pm\delta^{\pm1}\}\quad \text{ and } \quad 
\hat{\pi}(E_{\delta}^{*}[4])=\{0,\infty,\pm1,\pm i\},
$$ 
so $L_{\mathbb{Q}}\subseteq L_{\delta}$.
\begin{thm}
The field $L_{\delta}\backslash\{\infty\}$ is closed under taking
cube roots.
\end{thm}

\begin{proof}
We will show that if $a\in L_{\delta}\backslash\{\infty\}$, then
$\sqrt[3]{a-1}\subseteq L_{\delta}\backslash\{\infty\}$. This is
trivial for $a=1$. For $a=0$, since $L_{\delta}\backslash\{\infty\}$
contains all of the torsion points of $E_{\delta}$, it also contains
all of the roots of unity \cite[Page 96, Corollary 8.1.1]{MR2514094}.
For $a\neq0,1$, since $\pm\sqrt{a}\in L_{\delta}\backslash\{\infty\}$
by Theorem \ref{Theorem 3}, let us consider the elliptic curve
\[
E:y^{2}=(x-a)(x+\sqrt{a})(x-\sqrt{a})=x^{3}-ax^{2}-ax+a^{2}.
\]
The third (modified) division polynomial of $E$ is
\begin{eqnarray*}
F_{3}(x) & \coloneqq & {\displaystyle \prod_{\{\hat{\pi}(P):P\in E^{*}[3]\}}}(x-\hat{\pi}(P))\\
 & = & x^{4}-\dfrac{4}{3}ax^{3}-2ax^{2}+4a^{2}x-\dfrac{4}{3}a^{3}-\dfrac{1}{3}a^{2}\\
 & = & (x-\alpha_{1})(x-\alpha_{2})(x-\alpha_{3})(x-\alpha_{4}).
\end{eqnarray*}
The resolvent cubic of $F_{3}(x)$ is
\begin{eqnarray*}
rc(x) & = & (x-(\alpha_{1}\alpha_{2}+\alpha_{3}\alpha_{4}))(x-(\alpha_{1}\alpha_{3}+\alpha_{2}\alpha_{4}))(x-(\alpha_{1}\alpha_{4}+\alpha_{2}\alpha_{3}))\\
 & = & x^{3}+2ax^{2}+\dfrac{4}{3}a^{2}x+\dfrac{64}{27}a^{5}-\dfrac{128}{27}a^{4}+\dfrac{8}{3}a^{3}.
\end{eqnarray*}
The roots of $rc(x)=0$ are $x=-\dfrac{2}{3}a-\dfrac{4}{3}a\cdot\sqrt[3]{(a-1)^{2}}$,
so $\sqrt[3]{a-1}\subseteq L_{\delta}\backslash\{\infty\}$.
\end{proof}
\begin{cor}
\label{Corollary 7}
All quadratic, cubic, and quartic equations are
solvable in $L_{\delta}\backslash\{\infty\}$.
\end{cor}

\begin{example}
For any $\delta\in L_{\mathbb{Q}}\backslash\{0,\infty,\pm1,\pm i\}$,
we have $L_{\delta}=L_{\mathbb{Q}}$, so $L_{\mathbb{Q}}$ is the
smallest $L_{\delta}$.
\end{example}

\begin{example}
In the constructions of Theorem \ref{Theorem 1} and Remark \ref{Remark 4},
we require that
$$
u\in\hat{\pi}(E_{\delta_{1}}^{*}[3])\cap\hat{\pi}(E_{\delta_{2}}^{*}[3]),
$$
which implies $\delta_{1}\delta_{2}=-1/u^{2}$. Thus $L_{\delta_{1}}=L_{\delta_{2}}$.
\end{example}

Another example shows that $L_{\delta}$ is an
isogeny invariant.
\begin{cor}
\label{Corollary 10}If $E_{\delta_{1}}$ and $E_{\delta_{2}}$ are
isogenous over $\bar{\mathbb{Q}}$, then $L_{\delta_{1}}=L_{\delta_{2}}$.
\end{cor}

\begin{proof}
By \cite[Page 74, Remark 4.13.2]{MR2514094}, the nonsingular model
of $E_{\delta_{2}}$ can be defined over $\mathbb{Q}(E_{\delta_{1}}[\infty])$,
which is a subfield of $L_{\delta_{1}}\backslash\{\infty\}$, so the
$j$-invariant of $E_{\delta_{2}}$,
\[
j(E_{\delta_{2}})=\frac{16\left(\delta_{2}^{4}+\delta_{2}^{-4}+14\right)^{3}}{\left(\delta_{2}^{4}+\delta_{2}^{-4}-2\right)^{2}}\in L_{\delta_{1}}\backslash\{\infty\}.
\]
Thus $\delta_{2}^{4}+\delta_{2}^{-4}$ can be obtained by solving
a cubic equation with coefficients in $\mathbb{Q}(j(E_{\delta_{2}}))$.
By Corollary \ref{Corollary 7}, $\delta_{2}^{4}+\delta_{2}^{-4}\in L_{\delta_{1}}\backslash\{\infty\}$,
and then $\delta_{2}\in L_{\delta_{1}}\backslash\{\infty\}$.
\end{proof}
\begin{cor}
The field $L_{\delta}\backslash\{\infty\}$ is Galois over $\mathbb{Q}(j(E_{\delta}))$.
In particular, the field $L_{\mathbb{Q}}\backslash\{\infty\}$ is
Galois over $\mathbb{Q}$.
\end{cor}

\begin{proof}
Take $\sigma\in G_{\mathbb{Q}(j(E_{\delta}))}$ and $a\in L_{\delta}\backslash\{\infty\}$,
then $\sigma(a)\in L_{\sigma(\delta)}\backslash\{\infty\}$. Since
$j(E_{\sigma(\delta)})=\sigma(j(E_{\delta}))=j(E_{\delta})$, Corollary
\ref{Corollary 10} implies $L_{\sigma(\delta)}=L_{\delta}$.
\end{proof}
\begin{cor}
If $E$ is an elliptic curve in the Weierstrass form and defined over
a number field $\mathbb{Q}(j(E))$, then $L[\hat{\pi}(E[2])]\backslash\{\infty\}$
is a Galois field extension of $\mathbb{Q}(j(E))$.
\end{cor}

\begin{proof}
Take $\delta$ such that $j(E)=j(E_{\delta})\in L_{\delta}$. By Corollary
\ref{Corollary 7}, $\hat{\pi}(E[2])\subseteq L_{\delta}$, so the
linear fractional transformation between $L[\hat{\pi}(E[2])]$ and
$L_{\delta}$ is defined over $L_{\delta}\backslash\{\infty\}$. Therefore,
$L[\hat{\pi}(E[2])]=L_{\delta}$.
\end{proof}
\textbf{Acknowledgments.} The authors are grateful to Prof. Kevin
Ford for indicating Remark \ref{Remark 5}. The first author was partially
supported by the Russian Academic Excellence Project `5-100', Simons
Fellowship, and EPSRC programme grant EP/M024830. The second author
was supported by the MacCracken Program offered by New York University.

\medskip{}

Fedor Bogomolov\\
Courant Institute of Mathematical Sciences, New York University\\
251 Mercer Street, New York, NY 10012, USA\\
Email: bogomolo@cims.nyu.edu

\medskip{}

Fedor Bogomolov\\
Laboratory of Algebraic Geometry and its Applications\\
National Research University Higher School of Economics\\
7 Vavilova Street, 117312 Moscow, Russia

\medskip{}

Hang Fu\\
Courant Institute of Mathematical Sciences, New York University\\
251 Mercer Street, New York, NY 10012, USA\\
Email: fu@cims.nyu.edu
\end{document}